\documentclass[12pt, a4paper]{amsart}
\usepackage{amsmath,amsfonts,amssymb,amsthm, color}
\usepackage[numeric,abbrev,nobysame]{amsrefs}
\usepackage[all]{xy}
\usepackage{hyperref}\usepackage{soul}
\allowdisplaybreaks

\begin{document}

\newtheorem{thm}{Theorem}[section]
\newtheorem{prop}[thm]{Proposition}
\newtheorem{coro}[thm]{Corollary}
\newtheorem{conj}[thm]{Conjecture}
\newtheorem{example}[thm]{Example}
\newtheorem{lem}[thm]{Lemma}
\newtheorem{rem}[thm]{Remark}
\newtheorem{hy}[thm]{Hypothesis}
\newtheorem*{acks}{Acknowledgements}
\theoremstyle{definition}
\newtheorem{de}[thm]{Definition}
\xymatrixcolsep{5pc}

\newcommand{\C}{{\mathbb{C}}}
\newcommand{\Z}{{\mathbb{Z}}}
\newcommand{\N}{{\mathbb{N}}}
\newcommand{\Q}{{\mathbb{Q}}}
\newcommand{\te}[1]{\textnormal{{#1}}}
\newcommand{\set}[2]{{
    \left.\left\{
        {#1}
    \,\right|\,
        {#2}
    \right\}
}}
\newcommand{\sett}[2]{{
    \left\{
        {#1}
    \,\left|\,
        {#2}
    \right\}\right.
}}

\newcommand{\choice}[2]{{
\left[
\begin{array}{c}
{#1}\\{#2}
\end{array}
\right]
}}
\def \<{{\langle}}
\def \>{{\rangle}}

\def\({\left(}

\def\){\right)}

\newcommand{\overit}[2]{{
    \mathop{{#1}}\limits^{{#2}}
}}
\newcommand{\belowit}[2]{{
    \mathop{{#1}}\limits_{{#2}}
}}

\newcommand{\wt}[1]{\widetilde{#1}}

\newcommand{\wh}[1]{\widehat{#1}}

\newcommand{\no}[1]{{
    \mathopen{\overset{\circ}{
    \mathsmaller{\mathsmaller{\circ}}}
    }{#1}\mathclose{\overset{\circ}{\mathsmaller{\mathsmaller{\circ}}}}
}}

\newlength{\dhatheight}
\newcommand{\dwidehat}[1]{%
    \settoheight{\dhatheight}{\ensuremath{\widehat{#1}}}%
    \addtolength{\dhatheight}{-0.45ex}%
    \widehat{\vphantom{\rule{1pt}{\dhatheight}}%
    \smash{\widehat{#1}}}}
\newcommand{\dhat}[1]{%
    \settoheight{\dhatheight}{\ensuremath{\hat{#1}}}%
    \addtolength{\dhatheight}{-0.35ex}%
    \hat{\vphantom{\rule{1pt}{\dhatheight}}%
    \smash{\hat{#1}}}}

\newcommand{\dwh}[1]{\dwidehat{#1}}


\newcommand{\g}{{\frak g}}
\newcommand{\gc}{{\bar{\g'}}}
\newcommand{\h}{{\frak h}}
\newcommand{\cent}{{\frak c}}
\newcommand{\notc}{{\not c}}
\newcommand{\Loop}{{\cal L}}
\newcommand{\G}{{\cal G}}
\newcommand{\al}{\alpha}
\newcommand{\alck}{\al^\vee}
\newcommand{\be}{\beta}
\newcommand{\beck}{\be^\vee}
\newcommand{\ssl}{{\mathfrak{sl}}}

\newcommand{\rtu}{{\xi}}
\newcommand{\period}{{N}}
\newcommand{\half}{{\frac{1}{2}}}
\newcommand{\quar}{{\frac{1}{4}}}
\newcommand{\oct}{{\frac{1}{8}}}
\newcommand{\hex}{{\frac{1}{16}}}
\newcommand{\reciprocal}[1]{{\frac{1}{#1}}}
\newcommand{\inverse}{^{-1}}
\newcommand{\SumInZm}[2]{\sum\limits_{{#1}\in\Z_{#2}}}
\newcommand{\uce}{{\mathfrak{uce}}}
\newcommand{\I}{\textbf{\textrm{i}}}


\newcommand{\orb}[1]{|\mathcal{O}({#1})|}
\newcommand{\up}{_{(p)}}
\newcommand{\uq}{_{(q)}}
\newcommand{\upq}{_{(p+q)}}
\newcommand{\uz}{_{(0)}}
\newcommand{\uk}{_{(k)}}
\newcommand{\nsum}{\SumInZm{n}{\period}}
\newcommand{\ksum}{\SumInZm{k}{\period}}
\newcommand{\overN}{\reciprocal{\period}}
\newcommand{\df}{\delta\left( \frac{\xi^{k}w}{z} \right)}
\newcommand{\dfl}{\delta\left( \frac{\xi^{\ell}w}{z} \right)}
\newcommand{\ddf}{\left(D\delta\right)\left( \frac{\xi^{k}w}{z} \right)}

\newcommand{\ldfn}[1]{{\left( \frac{1+\xi^{#1}w/z}{1-{\xi^{#1}w}/{z}} \right)}}
\newcommand{\rdfn}[1]{{\left( \frac{{\xi^{#1}w}/{z}+1}{{\xi^{#1}w}/{z}-1} \right)}}
\newcommand{\ldf}{{\ldfn{k}}}
\newcommand{\rdf}{{\rdfn{k}}}
\newcommand{\ldfl}{{\ldfn{\ell}}}
\newcommand{\rdfl}{{\rdfn{\ell}}}

\newcommand{\kprod}{{\prod\limits_{k\in\Z_N}}}
\newcommand{\lprod}{{\prod\limits_{\ell\in\Z_N}}}
\newcommand{\E}{{\mathcal{E}}}
\newcommand{\F}{{\mathcal{F}}}

\newcommand{\Etopo}{{\mathcal{E}_{\te{topo}}}}

\newcommand{\Ye}{{\mathcal{Y}_\E}}

\newcommand{\rh}{{{\bf h}}}
\newcommand{\rp}{{{\bf p}}}
\newcommand{\rrho}{{{\pmb \varrho}}}
\newcommand{\ral}{{{\pmb \al}}}

\newcommand{\comp}{{\mathfrak{comp}}}
\newcommand{\ctimes}{{\widehat{\boxtimes}}}
\newcommand{\ptimes}{{\widehat{\otimes}}}
\newcommand{\ptimeslt}{{
{}_{\te{t}}\ptimes
}}
\newcommand{\ptimesrt}{{\ptimes_{\te{t}} }}
\newcommand{\ttp}[1]{{
    {}_{{#1}}\ptimes
}}
\newcommand{\bigptimes}{{\widehat{\bigotimes}}}
\newcommand{\bigptimeslt}{{
{}_{\te{t}}\bigptimes
}}
\newcommand{\bigptimesrt}{{\bigptimes_{\te{t}} }}
\newcommand{\bigttp}[1]{{
    {}_{{#1}}\bigptimes
}}

\newcommand{\ot}{\otimes}

\newcommand{\affva}[1]{V_{\wh\g}\(#1,0\)}
\newcommand{\saffva}[1]{L_{\wh\g}\(#1,0\)}
\newcommand{\saffmod}[1]{L_{\wh\g}\(#1\)}

\newcommand{\tar}{{\mathcal{DY}}_0\(\mathfrak{gl}_{\ell+1}\)}
\newcommand{\U}{{\mathcal{U}}}
\newcommand{\htar}{\mathcal{DY}_\hbar\(A\)}
\newcommand{\hhtar}{\widetilde{\mathcal{DY}}_\hbar\(A\)}
\newcommand{\htarz}{\mathcal{DY}_0\(\mathfrak{gl}_{\ell+1}\)}
\newcommand{\hhtarz}{\widetilde{\mathcal{DY}}_0\(A\)}
\newcommand{\qhei}{{\U_q\left( \hat{\h}_\mu \right)}}
\newcommand{\qheip}{{\U_q\left( \hat{\h}^+_\mu \right)}}
\newcommand{\qhein}{{\U_q\left( \hat{\h}^-_\mu \right)}}
\newcommand{\symalg}{{S\left( \hat{\h}^-_\mu \right)}}
\newcommand{\n}{{\mathfrak{n}}}
\newcommand{\vac}{{{\bf 1}}}
\newcommand{\vtar}{{{
    \mathcal{V}_{\hbar,\tau}\left(\ell,0\right)
}}}

\newcommand{\qtar}{
    \U_q\(\wh\g_\mu\)}
\newcommand{\rk}{{\bf k}}
\newcommand{\hn}{\U_\hbar\(\wh{\mathfrak n}\)}
\newcommand{\hnt}{\U_\hbar^\eta\(\wh{\mathfrak n}_\mu\)}
\newcommand{\hnta}{\U_\hbar^{\eta'}\(\wh{\mathfrak n}_\mu\)}
\newcommand{\qn}{\U_q\(\wh{\mathfrak n}_\mu\)}
\newcommand{\hvat}{V_{\hbar}^\eta\(\wh{\mathfrak n}_\mu\)}
\newcommand{\whvat}{\wt V_{\hbar}^\eta\(\wh{\mathfrak n}_\mu\)}


\newcommand{\hctvs}[1]{Hausdorff complete linear topological vector space}
\newcommand{\hcta}[1]{Hausdorff complete linear topological algebra}
\newcommand{\ons}[1]{open neighborhood system}
\newcommand{\B}{\mathcal{B}}
\newcommand{\rx}{{\bf x}}
\newcommand{\re}{{\bf e}}
\newcommand{\rphi}{{\boldsymbol{ \phi}}}

\newcommand{\der}{\mathcal D}

\newcommand{\chg}{\check{\mathfrak g}}
\newcommand{\chh}{\check{\mathfrak h}}
\newcommand{\cha}{\check{a}}
\newcommand{\chb}{\check{b}}
\newcommand{\ka}{\mathfrak{a}}
\newcommand{\chka}{\check{\mathfrak{a}}}


\makeatletter
\renewcommand{\BibLabel}{%
    \Hy@raisedlink{\hyper@anchorstart{cite.\CurrentBib}\hyper@anchorend}%
    [\thebib]%
}
\@addtoreset{equation}{section}
\def\theequation{\thesection.\arabic{equation}}
\makeatother \makeatletter

\title[Twisted quantum affinizations]{Twisted quantum affinizations and their vertex representations}

\author{Fulin Chen$^1$}
\address{Department of Mathematics, Xiamen University,
 Xiamen, China 361005} \email{chenf@xmu.edu.cn }\thanks{$^1$Partially supported by NSF of China (No. 11501478)  and Fundamental Research Funds for the Central University (No. 20720150003).}

\author{Naihuan Jing$^2$}
\address{Department of Mathematics, North Carolina State University, Raleigh, NC 27695,
USA}
\email{jing@math.ncsu.edu}
\thanks{$^2$Partially supported by NSF of China (No. 11531004) and Simons Foundation (No. 523868).}

\author{Fei Kong$^3$}
\address{Department of Mathematics, South China University of Technology, Guangzhou, China 510640} \email{kfkfkfc@scut.edu.cn}
\thanks{$^3$Partially supported by China Postdoctoral Science Foundation (No. 2016M602454) and NSF of China (No. 11701183).}

 \author{Shaobin Tan$^4$}
 \address{Department of Mathematics, Xiamen University,
 Xiamen, China 361005} \email{tans@xmu.edu.cn}
 \thanks{$^4$Partially supported by NSF of China (Nos. 11471268 and 11531004).}


\begin{abstract}
In this paper we generalize Drinfeld's twisted quantum affine algebras to construct twisted
quantum algebras for all simply-laced generalized Cartan matrices
and present their vertex representation realizations. 

\end{abstract}
\maketitle

\section{Introduction}

Quantum (twisted and untwisted) affine algebras are one of the most important subclasses of quantum Kac-Moody algebras.
They were first introduced by Drinfeld and Jimbo
in terms of the Chevalley generators and Serre relations. To classify their finite dimensional representations,
Drinfeld later gave new realization for the quantum affine algebra 
as analogue of the affine Kac-Moody algebras \cite{Dr-new}. 
Drinfeld realization has since played a fundamental role in quantum conformal field theory \cite{FJ-vr-qaffine, Fr, JimboM}.

Drinfeld's untwisted quantum affinization process had been extended to all symmetrizable quantum enveloping Kac-Moody algebras.
The first example  was given in \cite{GKV}, in which  the quantum affinization
$\U_q(\dwidehat{\ssl}_{\ell+1})$ of $\U_q(\widehat{\ssl}_{\ell+1})$ was  introduced.
Moreover, an additional parameter $p$ can be added in the quantum affinization process \cite{GKV} and then one gets a two parameter deformed algebra
$\U_{q,p}(\dwidehat{\ssl}_{\ell+1})$.
For the general quantum Kac-Moody algebras
$\U_q(\g)$, their untwisted quantum affinization  $\U_q(\wh\g)$ were defined by Jing \cite{J-KM}, Nakajima \cite{Naka-quiver}
and Hernandez \cite{He-representation-coprod-proof}.
It is notable 
that the representation theory of quantum toroidal algebras (untwisted quantum affinization of
untwisted quantum affine algebras) is very rich, see the survey \cite{He-total} for further comprehensive studies.

As we have mentioned above, for the purpose of giving a current algebra realization of twisted quantum affine algebras,
\cite{Dr-new} Drinfeld has also introduced twisted quantum affine algebras
associated to diagram automorphisms on finite dimensional simple Lie algebras. 
It is natural to ask whether twisted quantum affinization can be generalized to diagram automorphisms for the general
simply-laced Kac-Moody algebras?

The goal of this paper is to answer the above question in
the most general form. In Sect. 2, we first introduce a class of diagram automorphisms on
the simply-laced Kac-Moody algebra $\g$ with certain natural linking condition.
To such a diagram automorphism 
$\mu$,
 we define a new quantum algebra $\qtar$ for $\g$. 
When $\g$ is  of finite type, the algebra $\qtar$ is the twisted quantum affinization  introduced by Drinfeld;
when $\mu=1$, it coincides with the untwisted quantum affine algebra $\mathcal U_q(\hat{\g})$.

To show our quantum algebra $\qtar$ 
is nontrivial in general, the vertex representations of $\qtar$
 are presented in Sect.3 (see Theorem \ref{thm:vertex-repn}) and proved in Sect.4.
Our construction generalizes some of the important
techniques introduced by Lepowsky in studying twisted vertex operators \cite{Lep}, and
is a common generalization of those vertex representations obtained in \cite{FJ-vr-qaffine,J-inv,J-KM}.

Another motivation of this paper comes from the quantization theory of extended affine Lie algebras.
Recall that (the core of) extended affine Lie algebras form a relatively large family of Lie algebras, including
 finite dimensional simple Lie algebras, (twisted and untwisted) affine Lie algebras, toroidal Lie algebras, and quantum
torus Lie algebras $\dwidehat{\ssl}_{\ell+1}(\C_p)$ (see \cite{BGK},\cite{AABGP}, \cite{N2} and the references
therein). The quantization of these special extended affine Lie algebras, namely, quantum algebras of finite type, (twisted and untwisted) quantum affine algebras, quantum toroidal algebras and quantum algebras $\U_{q,p}(\dwidehat{\ssl}_{\ell+1})$,
have been intensively studied.
It was proved in \cite{ABP} that the core of a  nullity $2$ extended affine Lie algebra is either isomorphic to $\dwidehat{\ssl}_{\ell+1}(\C_p)$   or 
to the twisted affinization of an affine Lie algebra.
We believe that the twisted quantum affinization of quantum affine algebras introduced in this paper are the ``right" quantization of
the  twisted affinization of affine Lie algebras.
Therefore our twisted quantum affinization of the quantum affine algebra
also solves the question of 
quantizing the twisted affinization of the affine Lie algebra. This process is believed to be the
way how one can obtain quantization of all nullity $2$ extended affine Lie algebras.

Throughout this paper, the sets of integers, non-negative integers and complex numbers will be denoted respectively by $\Z$, $\N$,  and $\C$.
We let $q$ be a fixed generic nonzero complex number,
and for $n\in\Z$, set
$[n]_q=\frac{q^n-q^{-n}}{q-q\inverse}$.

\section{Twisted quantum affinizations}
In this section we introduce the notion of general twisted quantum affinizations associated to a class of automorphisms on simply-laced generalized Cartan matrices.

Let $\nu$ be a positive integer
and let $A=(a_{ij})_{i,j=1}^\nu$ be a simply-laced generalized Cartan matrix (GCM). By definition, a GCM is a $\nu\times\nu$-matrix such that
\[a_{ij}\in \Z,\ a_{ii}=2,\ i\neq j\Rightarrow a_{ij}=0,-1,\ a_{ij}=0\Leftrightarrow a_{ji}=0,\]
for $i,j=1,\cdots,\nu$.
An automorphism $\mu$ of $A$ is a permutation of the index set $I=\{1,2,\cdots,\nu\}$ such that
 $a_{ij}=a_{\mu(i)\mu(j)}$ for all $i,j\in I$. Assume that $\mu$ is of order $N$ and let $\Z_N=\Z/N\Z$ be the cyclic group of order $N$.
 For $i,j\in I$, set
\begin{align*}\label{gammaij}
\Gamma_{ij}=\{k\in \Z_N\mid a_{i\mu^k(j)}\ne 0\},\quad
\Gamma_{ij}^\pm=\{k\in \Z_N\mid \pm a_{i\mu^k(j)}>0\}.\end{align*}
Throughout this paper, we assume that $\mu$  satisfies  the following  linking condition
\begin{align*}
\te{(LC) }\quad \te{for every}\ i,j\in I\ \te{with}\ a_{ij}<0,\ \Gamma_{ij}^-\te{ is a subgroup of }\Z_N.\qquad\qquad\qquad
\end{align*}

The linking condition is a natural extension of the common properties shared by Dynkin diagrams of finite and affine types. 
\begin{lem} Assume that $A$ is of finite or affine type. Then every automorphism $\mu$ of $A$ satisfies the
condition (LC) except that $A$ is of type $A_{\ell}^{(1)}$, $\ell=2$ or $\ell \ge 4$ and $\mu$
is an order $\ell+1$ cyclic rotation of the Dynkin diagram. 
\end{lem}
\begin{proof} The lemma is directly checked for all possible $A$ and
the diagram automorphisms of $A$ (see \cite[Chapter 4]{Kac-book} for details).
\end{proof}

For  $i,j\in I$, we set
\begin{align*}d_{ij}=\mathrm{Card}\ \Gamma_{ij}^-,\quad d_i=\mathrm{Card}\ \Gamma_{ii}^+.\end{align*}
Note that both $d_i$ and $d_j$ divide $d_{ij}$. Let $\xi$ be a fixed $N$-th primitive root of unity.

\begin{lem}\label{lem:s-i=2}
If $d_{ii}>0$, then one has $\prod_{k\in\Gamma_{ii}^-}\xi^k=-1$.
\end{lem}

\begin{proof}
For $k\in\Z$, we set $\bar k=k+N\Z\in\Z_N$.
Let $k_0$ be the minimal positive integer such that $\bar {k_0}\in \Gamma_{ii}^-$.
We remark that $\Gamma_{i\mu^{k_0}(i)}^-$ is a subgroup of $\Z_N$ and $\mathrm{Card} \Gamma_{i\mu^{k_0}(i)}^-=\mathrm{Card} \Gamma_{ii}^-=d_{ii}$.
Therefore one gets that
\begin{align*}
\Gamma_{ii}^-=\set{\overline{k_0+pN/d_{ii}}}{p=0,\dots,d_{ii}-1}.
\end{align*}
Since $\overline{k_0-N/d_{ii}}\in\Gamma_{ii}^-$ and $\bar 0\not\in\Gamma_{ii}^-$, the minimality of $k_0$ forces that $k_0<N/d_{ii}$.
Moreover, it follows from $\overline{N-k_0}\in\Gamma_{ii}^-$ that $N-k_0=k_0+ (d_{ii}-1)N/d_{ii}$.
This implies that $d_{ii}=N/(2k_0)$ and hence
\begin{align*}
&\prod_{k\in\Gamma_{ii}^-}\xi^k=\xi^{k_0d_{ii}}=\xi^{k_0\frac{N}{2k_0}}=-1.
\end{align*}
\end{proof}

Let $(\h,\Pi,\Pi^\vee)$ be a realization of the generalized Cartan matrix $A$ (see \cite[Chapter 1]{Kac-book}). i.e.  $\h$ is a $(2\nu-\mathrm{rank}(A))$-dimensional $\C$-vector space,
$\Pi=\{\al_i\}_{i\in I}$ is a set of linearly independent elements in $\h^\ast$, $\Pi^\vee=\{\al_i^\vee\}_{i\in I}$ is
a set of linearly independent elements in $\h$ and $\al_j(\al_i^\vee)=a_{ij}$ for $i,j\in I$.
Let $\g=\g(A)$ be the Kac-Moody algebra associated to the quadruple $(A,\h,\Pi,\Pi^\vee)$.
It is known that $\mu$ can be lifted to an automorphism of $\g$ with order $N$, called a diagram automorphism (\cite{KW}).

Let $Q=\oplus_{i\in I}\Z\al_i$ be the root lattice of $\g$ equipped with the symmetric $\Z$-bilinear form $\<\cdot|\cdot\>$ such that
\begin{align}
\<\al_i|\al_j\>=a_{ij},\quad i,j\in I.
\end{align}
By definition one can view $\mu$ as an isometry of $Q$ such that $\mu(\al_i)=\al_{\mu(i)}$ for $i\in I$.

Let $z,w,z_1,z_2,z_3,\cdots $ be mutually commuting independent formal variables.
For any $i,j\in I$,  we introduce the polynomials
\begin{align*}
&F^\pm_{ij}(z,w)=\prod_{k\in \Gamma_{ij}}\left( z-\xi^{k}q^{\pm a_{i\mu^k(j)}}w \right),\\
&G^\pm_{ij}(z,w)=\prod_{k\in \Gamma_{ij}}\left( q^{\pm a_{i\mu^k(j)}}z-\xi^{k}w \right),
\end{align*}
and the formal series
\begin{align*}
&g_{ij}(z)=G^+_{ij}(1,z)/F^+_{ij}(1,z)=\prod_{k\in \Gamma_{ij}} \frac{ q^{a_{i\mu^k(j)}}-\xi^k z}{1-\xi^kq^{ a_{i\mu^k(j)}}z},
\end{align*}
which are expanded for $|z|<1$. For each $i\in I$, write $\mathcal O(i)$ for the $\mu$-orbit of $I$  contains $i$.
In the case that $a_{ij}<0$ and $i\notin\mathcal O(j)$, we also introduce the polynomial
\begin{align*}\label{pij}
p_{ij}^\pm(z,w)=
    \left(
        z^{d_{ii}}+q^{\mp d_{ii}}w^{d_{ii}}
    \right)
    \frac{
        q^{\pm 2d_{ij}}z^{d_{ij}}-w^{d_{ij}}
    }{
        q^{\pm 2d_i}z^{d_{i}}-w^{d_i}
    }.
\end{align*}
Furthermore, for each $i\in I$ with $d_{ii}>0$, we set
\begin{align*}
p_i(z_1,z_2,z_3)=&\left(q^{\mp \frac{3}{2}d_{ii}}z_{1}^{d_{ii}}
        -(q^{\frac{d_{ii}}{2}}+q^{-\frac{d_{ii}}{2}})
            z_{2}^{d_{ii}}
        +q^{\pm \frac{3}{2}d_{ii}}
            z_{3}^{d_{ii}}\right)\\
        &\times \prod_{1\le a<b\le 3}\frac{(z_{a}^{d_{ii}}-z_{b}^{d_{ii}})
        (z_{a}^{d_{ii}}-q^{\mp2d_{ii}}z_{b}^{d_{ii}})}
        {(z_{a}^{d_i}-z_{b}^{d_i})
            (z_{a}^{d_i}-q^{\mp2d_i}z_{b}^{d_i})}.\end{align*}

Now we introduce the 
general twisted quantum affinization algebra associated to the automorphism $\mu$ of $A$.
\begin{de}\label{de:tqaffine}
The twisted quantum affinization algebra $\qtar$ is a unital associative algebra
over $\C$
generated by
\begin{eqnarray}\label{eq:tqagenerators}
\set{k_\al,\ h_{i,m},\  x^\pm_{i,n},\ q^{\pm\half c}}
{
   \al\in Q, i\in I, m\in\Z\setminus\{0\},n\in\Z
},
\end{eqnarray}
subject to the following relations written in terms of generating functions in $z$
\begin{align*}
 &\phi_i^\pm(z)=k_{\pm \al_i}\, \te{exp}
    \left(
        \pm (q-q\inverse)\sum\limits_{\pm m> 0}h_{i,m}z^{-m}
    \right),\\
&x^\pm_i(z)=\sum\limits_{m\in \Z} x^\pm_{i,m}z^{-m}.
\end{align*}
The relations are ($i,j\in I$, $\al,\be\in Q$)
\begin{align*}
&\te{(Q0)  }&&x^\pm_{\mu(i)}(z)=x^\pm_i(\xi\inverse z),
    \quad \phi^\pm_{\mu(i)}(z)=\phi^\pm_i(\xi\inverse z),
    \quad k_{\mu(\al)}=k_\al,\\
&\te{(Q1)  }&&q^{\pm\half c}q^{\mp\half c}=1,\ \
    q^{\pm\half c} \te{ are central},\\
&\te{(Q2)  }&&k_\al k_{\be}=k_{\al+\be},\ k_0=1,\ [\phi_i^\pm(z),\phi_j^\pm(w)]=0=[k_\al,\phi^\pm_i(z)],\\
&\te{(Q3)  }&&k_\al x_i^\pm(z) k_{-\al}=q^{\pm\sum_{k\in \Z_N}\<\al|\al_{\mu^k(i)}\>}x_i^\pm(z),\\
&\te{(Q4) }&& \phi^+_i(z)\phi^-_j(w)=\phi^-_j(w)\phi^+_i(z)
    g_{ij}(q^c w/z)\inverse g_{ij}(q^{-c} w/z),\\
&\te{(Q5) }&&\phi^+_i(z)x^\pm_j(w)=x^\pm_j(w)\phi^+_i(z)
    g_{ij}(q^{\mp\half c}w/z)^{\pm 1}
    ,\\
&\te{(Q6) }&& \phi^-_i(z)x^\pm_j(w)=x^\pm_j(w)\phi^-_i(z)
    g_{ji}(q^{\mp\half c}z/w)^{\mp 1},\\
&\te{(Q7)  }&&[x_i^+(z),x_j^-(w)]
=\frac{1}{q_i-q_i\inverse}
    \ksum\delta_{i,\mu^k(j)}\\
 &&&  \quad\times \Bigg(
        \phi_i^+(q^{-\half c}z)\delta\left(
            \frac{q^c\xi^kw}{z}
        \right)
        -
        \phi_i^-(q^{\half c}z)\delta\left(
            \frac{q^{-c}\xi^kw}{z}
        \right)
    \Bigg)\\
&\te{(Q8)  }&&F^\pm_{ij}(z,w)x^\pm_i(z)x^\pm_j(w)=
    G^\pm_{ ij}(z,w)x^\pm_j(w)x^\pm_i(z),\\
&\te{(Q9)  }&&\sum_{\sigma\in S_{2}}\Big\{
    p_{ij}^\pm(z_{\sigma(1)},z_{\sigma(2)})\big(x_i^\pm(z_{\sigma(1)})x_i^\pm(z_{\sigma(2)})
    x_j^\pm(w)
   -[2]_{q^{d_{ij}}}
   x_i^\pm(z_{\sigma(1)})x_j^\pm(w)x_i^\pm(z_{\sigma(2)})\\
&&& \qquad\quad   +x_j^\pm(w)x_i^\pm(z_{\sigma(1)})x_i^\pm(z_{\sigma(2)})\big)
\Big\}\ =0,\quad
    \te{if }a_{ij}<0\ \te{and}\ i\notin\mathcal O(j),\\
&\te{(Q10) }&&
\sum_{\sigma\in S_3}\bigg\{p_i(z_{\sigma(1)},z_{\sigma(2)},z_{\sigma(3)})\,x_i^\pm(z_{\sigma(1)})x_i^\pm(z_{\sigma(2)})x_i^\pm(z_{\sigma(3)})
\bigg\}=0,\qquad
    \te{if }d_{ii}> 0,
\end{align*}
where $\delta(z)=\sum_{n\in\Z}z^n$ is the usual $\delta$-function.
\end{de}

\begin{rem} When $A$ is  of finite type, the algebra $\qtar$ was first introduced by Drinfeld (\cite{Dr-new})
 to give a current algebra realization for the twisted quantum affine algebras.
 When $\mu=1$, the algebra $\qtar$ was the untwisted quantum affinization of $\U_q(\g)$ studied in \cite{GKV,J-KM,Naka-quiver,He-representation-coprod-proof,He-drinfeld-coproduct}.
\end{rem}

\begin{rem}
The relations (Q4)-(Q6) can be rewritten in the term of  $h_{i,m}$, $i\in I, m\ne 0$ as follows (see \cite[Sect. 2]{FJ-vr-qaffine} and \cite[Props. 3.3, 3.8]{J-inv})
\begin{align*}
&\te{(Q4$'$)}&&[h_{i,m},h_{j,n}]=\delta_{m+n,0}\frac{1}{m}
    \ksum \xi^{mk}[ma_{i\mu^k(j)}]_q[m]_{q^c},\  m,n\ne 0,\\
&\te{(Q5$'$)}&&[h_{i,m},x_{j,n}^\pm]
=\pm \frac{1}{m}\ksum \xi^{mk}[ma_{i\mu^k(j)}]_qq^{\mp\half mc} x_{j,m+n}^\pm,\ m>0,n\in\Z,\\
&\te{(Q6$'$)}&&[h_{i,m},x_{j,n}^\pm]
=\pm \frac{1}{m}\ksum \xi^{mk}[ma_{i\mu^k(j)}]_qq^{\pm\half mc} x_{j,m+n}^\pm,\ m<0,n\in\Z,
\end{align*}
where $i,j\in I$ and
$[m]_{q^c}=\frac{q^c-q^{-c}}{q-q^{-1}}$.
\end{rem}

\begin{rem} Assume now that $A$ is of finite type.
Let $\check{I}$ be a representative set of $I$ under the action of $\mu$, and let
$\check{A}=(\check{a}_{ij})_{i,j\in \check{I}}$ be the $\mu$-folded matrix of $A$,
where
\begin{align*}
 \check{a}_{ij}=\frac{s_i}{d_i}\sum_{k\in\Z_N}a_{\mu^k(i),j}\quad \te{and}\quad s_i=3-\frac{1}{d_i}\sum_{k\in\Z_N}a_{\mu^k(i),i}.
\end{align*}
In  \cite{Da1,Da2}, Damiani proved that the Drinfeld's quantum Serre relation (Q9) can be replaced by the following quantum Serre relation
\begin{align*}
&\te{(Q9$'$) }\sum_{\sigma\in S_{1-\check{a}_{ij} }}
\sum_{r=0}^{1-\check{a}_{ij}}\choice{1-\check{a}_{ij} }{r}_{q^{\frac{d_i}{s_i}}}(-1)^r
    x_i^\pm(z_{\sigma(1)})\cdots x_i^\pm(z_{\sigma(r)})x_j^\pm(w)\\
&\qquad\qquad\times
    x_i^\pm(z_{\sigma(r+1)})\cdots x_i^\pm(z_{\sigma(1-\check{a}_{ij} )})=0,
    \qquad\te{if }\check{a}_{ij}<0.
\end{align*}
This is the crucial step in Damiani's proof on the isomorphism between Drinfeld-Jimbo's and Drinfeld's realizations of
twisted affine quantum algebras.

For the general case, the folded matrix $\check{A}$ may be not  a GCM again (and so
the relation (Q9\,$'$) may not exist).
It was proved in \cite{FSS} that $\check{A}$ is a GCM if $s_i\le 2$ for all $i\in I$.
In this case, it is expected that the  Serre relation (Q9) is equivalent to
the relation (Q9\,$'$) in $\qtar$. In view of Damiani's work, this is true if $\check{a}_{ij}\ge -3$ for all $i,j\in \check{I}$.
However, it seems that the method developed in \cite{Da1} cannot be applied directly to the cases  $\check{a}_{ij}\le -4$.
\end{rem}

\begin{rem} The quantum Serre relation (Q10) is new when $d_i\ne d_{ii}$ (or equivalently, when $s_i\ne 2$).
The correct function $p_i(z_1,z_2,z_3)$ is found in the computation of the Serre relation among quantum vertex operators (see Proposition \ref{prop:qserre2}).
\end{rem}

\begin{rem} In \cite{He-representation-coprod-proof}, the author proved a triangular decomposition of the untwisted
quantum affinization algebras, which plays a fundamental  role in their representation theory.
By the combinational identities proved in \cite[\S 3.3.3]{He-representation-coprod-proof}, one can prove that
the algebra $\qtar$ admits a similar triangular decomposition.
The details will be given in a forthcoming work, where the twisted quantum affinization of non-simply-laced quantum Kac-Moody
algebras are introduced and their structure theory is studied.
We want to empathize here that the Serre relation (Q9) is the ``correct '' one (rather than the relation (Q9\,$'$) that was
 needed in deducing the triangular decomposition of $\qtar$.
 \end{rem}

\section{Vertex representation of $\qtar$}\label{subset:def-of-fock-sp}
In this section we construct the vertex representation of $\qtar$.

Let $\qhei$ be the $\mu$-twisted quantum Heisenberg algebra associated to the root lattice $Q$.
It is the associative algebra generated by $\al_{i,m},$, $i\in I,0\ne m\in \Z$ and $q^{\pm c}$
subject to relations
\begin{align*} \te{(H0)}&\quad \al_{\mu(i),m}=\xi\al_{i,m},\\
\te{(H1)}&\quad [\al_{i,m},\al_{j,n}]=\delta_{m+n,0}\frac{1}{m}
    \ksum \xi^{mk}[m\< \al_i|\mu^k(\al_j)\>]_q[m]_{q^c},
\end{align*}
for $i,j\in I$ and $m,n\in \Z\setminus\{0\}$.
Let $\qhein$ be the subalgebra of $\qhei$ generated by the elements
$\{\al_{i,n}\,|\,i\in I,n<0\}$, which is
 isomorphic to the symmetric algebra $\symalg$ generated by the same elements.
We define a $\qhei$-action on $\symalg$ by the following rules
\begin{eqnarray*}
&&q^{\pm c}.1=q^{\pm 1},\quad \al_{i,-n}=\te{multiplication operator},\  i\in I, n>0,\\
&&\al_{i,n}=\te{annihilation operator subject to (H1)},\  i\in I, n>0.
\end{eqnarray*}
From now on, we take $\al_{i,n}$ ($i\in I$, $n\in\Z\setminus\{0\}$), $q^{\pm c}$
as operators on $\symalg$.
For each $i\in I$, we introduce the following fields on $\symalg$
\begin{align*}
&\al_i^\pm(z)=\sum\limits_{\pm m>0}\al_{i,m}z^{-m},\quad
E_\pm(\al_i,z)=\te{exp}\left(\mp\sum\limits_{m>0}
    \frac{\al_{i,\pm m}}{[m]_q}z^{\mp m}\right).
\end{align*}

We define a bilinear form on $Q$ as follows
\begin{eqnarray*}\label{eq:group-commutator}
C: Q\times Q\rightarrow \<\xi'\>,\quad (\al,\be)\mapsto \kprod(-\xi^{-k})^{\<\al|\mu^k(\be)\>}.
\end{eqnarray*}
 where $\xi'=(-1)^N\xi$.
Following \cite[Sect. 5]{Lep},
there is  a (unique up to isomorphism) central extension
\[\xymatrix{
  1 \ar[r] & \<\xi'\> \ar[r] & \hat{Q} \ar[r]^{-} & Q \ar[r] & 1 }\]
of $Q$ by the cyclic group $\<\xi'\>$ with the commutator map
\[aba^{-1}b^{-1}=C(\bar{a},\bar{b}),\quad a,b\in \hat{Q}.\]
We notice that the argument in \cite[Sect. 5]{Lep} remains valid for possibly degenerate lattices.
Thus,  as shown in \cite{Lep},  there is a lifting automorphism $\hat{\mu}$ on $\hat{Q}$  such that
\begin{eqnarray}\label{eq:lift-mu}
\hat{\mu}(a)=a\ \ \te{if }\mu(\bar{a})=\bar{a},\qquad  (\hat{\mu}(a))^-=\mu(\bar{a}),\ a\in \hat{Q}.
\end{eqnarray}

For each $\al\in Q$, write
\begin{align*}
\al_{(0)}=\al+\mu(\al)+\cdots+\mu^{N-1}(\al)\in Q.
\end{align*}
We denote by $Q_{(0)}$ the subgroup of $Q$ spanned by $\al_{(0)}$, $\al\in Q$.
In the remaining part of this section, we fix a $Q_{(0)}$-graded $\hat{Q}$-module $T=\oplus_{\gamma\in Q_{(0)}}T_\gamma$ on which the following
compatible conditions are satisfied
\begin{align}\label{eq:hatmuOnMod}
a.t\in T_{(\bar{a}+\be)_{(0)}},\quad \hat{\mu}(a).t
=\xi^{-\<\bar{a}_{(0)},\be\>-\<\bar{a}_{(0)}|\bar{a}\>/2}a.t,
\end{align}
for $a\in \hat{Q},\,t\in T_{\beta_{(0)}},\,\be\in Q$.
When $A$ is of finite type, such irreducible $\hat{Q}$-modules were classified and constructed explicitly in \cite[Sect. 6]{Lep}.
 We remark that such  $Q_{(0)}$-graded $\hat{Q}$-modules also exist for the general case.
For example, one can take $T$ to be the induced $\hat{Q}$-module $\C[\hat{Q}]\otimes_{\C[\hat{M}]}\tau$
with the
$Q_{(0)}$-gradation given by
\[T_{\gamma}=\,\te{span}_\C\{a\otimes 1\mid\bar{a}_{(0)}=\gamma, a\in \hat{Q}\},\quad \gamma\in Q_{(0)}.\]
Here,  $\hat{M}$ stands for  the pulling back of $(1-\mu)Q$ in $\hat{Q}$, and $\tau$ stands for
the character  on $\hat{M}$ given by
\begin{eqnarray*}
\tau(\xi')=\xi',\quad  \tau(a\hat{\mu}a\inverse)=\xi^{-\<\bar{a}_{(0)}|\bar{a}\>/2},\quad a\in \hat{Q}.
\end{eqnarray*}
The fact that $\tau$ is a character follows from \eqref{eq:lift-mu}.

For  $\al\in Q$ and $0\ne c\in \C$,
we define the operators $z^{\al_{(0)}}\in \te{End}(T)[[z,z^{-1}]]$ and $c^{\al_{(0)}}\in \te{End}(T)$ as follows
\begin{eqnarray*}
z^{\al_{(0)}}.t=z^{\<\al_{(0)}|\be\>}t,\quad c^{\al_{(0)}}.t=c^{\<\al_{(0)}|\be\>}t, \quad t\in T_{\be_{(0)}}, \be\in Q.
\end{eqnarray*}
For convenience, we fix a section $e:\al\mapsto e_\alpha$ from $Q$ to $\hat{Q}$.

The generalized Fock space $F_T$ is defined as the tensor product of the $\qhei$-module $\symalg$ and the $\hat{Q}$-module $T$.
For each $i\in I$, let us introduce the following twisted vertex operators on the generalized Fock space $F_T$
\begin{eqnarray*}
&\Phi_i^\pm(z)=\sum_{m\ge 0}\Phi_{i,m}^\pm z^{-m}=q^{\pm\al_{i(0)}}\te{exp}\left(\pm (q-q\inverse)\al_i^\pm(z)\right),\\
&X_i^\pm(z)=\sum_{m\in \Z}X^\pm_{i,m}z^{-m}=E_-(\pm\al_i,q^{\mp\half}z)E_+(\pm\al_i,q^{\pm\half}z)e_{\pm\al_i}
    z^{\pm\al_{i(0)}+\<\al_{i(0)}|\al_i\>/2}.
\end{eqnarray*}
Here and henceforth, the operators  $\al_{i,n}, n\ne 0$, $e_{\al_i}$, $z^{\al_{i(0)}}$ and $c^{\al_{i(0)}} (c\in \C(q)^*)$  are always viewed
  as operators on  $F_T=\symalg\otimes T$ in the following natural way
  \[\al_{i,n}=\al_{i,n}\otimes 1,\,
  e_{\al_i}=1\otimes e_{\al_i},\, z^{\al_{i(0)}}=1\otimes z^{\al_{i(0)}},\, c^{\al_{i(0)}}=1\otimes c^{\al_{i(0)}}.\]

When $\mu=1$, the $\wh{Q}$-module $T$ can be taken to be the group algebra of $Q$. In this case, the vertex operators
$\Phi_i^\pm(z), X_i^\pm(z), i\in I$ were first introduced in \cite{FJ-vr-qaffine} for the case of finite type and then
in \cite{J-KM} for the case of general type.
It was proved in \cite[Theorem 3.1]{J-KM} that  these vertex operators provide a realization of the untwisted
quantum affinization algebras.

Just as the untwisted case, we have the following main result of this paper, whose proof will be given in the next section.
\begin{thm}\label{thm:vertex-repn}
The generalized Fock space $F_T$  affords a representation for the twisted quantum affinization algebra  $\qtar$ with actions given by
\begin{eqnarray*}
&&q^{\pm\half c}\mapsto q^{\pm\half},\quad
\phi_{i,m}^\pm\mapsto \Phi^\pm_{i,m},\quad
x^\pm_{i,m}\mapsto \epsilon_i X^\pm_{i,m},
\end{eqnarray*}
where $\ i\in I, m\in \Z$ and
\begin{equation}\label{eplsioni}
\epsilon_i=\begin{cases}\left(
    \frac{d_i\left[
        {d_i}
    \right]_q}{[2]_{q^{d_{ii}/2}}}
\right)^\half,\ &\te{if} \ d_{ii}>0,\\
\left(
    d_i\left[
        {d_i}
    \right]_q
\right)^\half,\ &\te{if}\ d_{ii}=0.
\end{cases}\end{equation}
\end{thm}

\section{Proof of Theorem \ref{thm:vertex-repn}}
This section is devoted to prove Theorem \ref{thm:vertex-repn}.
For every $a\in \Z$, recall the following notation introduced in \cite{J-Z-alg}
\begin{eqnarray}\label{eq:q2notion}
&&(1-z)_{q^2}^a=\frac{
\prod\limits_{n\geq 0}(1-q^{-a+1+2n}z)
}{
\prod\limits_{n\geq 0}(1-q^{a+1+2n}z)
}.
\end{eqnarray}
We will also write $(z-w)_{q^2}^a=z^{a}(1-w/z)_{q^2}^a$.
The notation ``$(1-z)_{q^2}^a$" is a $q$-deformation of $(1-z)^a$ in the following sense.
\begin{lem}\label{lem:qibinormal}
For every $a\in\Z$, one has
\begin{eqnarray*}\label{eq:qibinormal}
-\sum\limits_{m>0}\frac{1}{m}[a]_{q^m}z^{m}=\te{log}(1-z)_{q^2}^a.
\end{eqnarray*}
\end{lem}

\begin{proof}
It follows from the fact $[a]_q=\sum\limits_{n\geq 0}\left(q^{-a+1+2n}- q^{a+1+2n} \right)$
that
\begin{eqnarray*}
&&\sum\limits_{m>0}\frac{1}{m}[a]_{q^m}z^{m}
=\sum\limits_{m>0}\frac{1}{m}\sum\limits_{n\geq 0}\left(q^{(-a+1+2n)m}- q^{(a+1+2n)m} \right)z^{m}\nonumber\\
&&\quad=-\sum\limits_{n\geq 0}\left(
    \te{log}\left(1-q^{-a+1+2n}z\right)
    -
    \te{log}\left(1-q^{a+1+2n}z\right)
\right)\nonumber\\
&&\quad=-\te{log}
\left(
    \frac{
        \prod\limits_{n\geq 0}\left(1-q^{-a+1+2n}z\right)
    }{
        \prod\limits_{n\geq 0}\left(1-q^{a+1+2n}z\right)
    }
\right)
=-\te{log}(1-z)_{q^2}^a.
\end{eqnarray*}
\end{proof}

We first give some elementary properties about the operators $E_\pm(\al_i,z)$.
\begin{lem}\label{lem:alcom}
For any $i,j\in I$, one has

$\te{(1) }\al_{\mu(i)}^\pm(z)=\al_i^\pm(\xi\inverse z),\ E_\pm(\al_{\mu(i)},z)=E_\pm(\al_i,\xi\inverse z)$,

$\te{(2) }[\al_i^+(z),E_-(\al_j,w)]=\frac{1}{q-q\inverse}\te{log}
    \left(q^{-\<\sum\al_{\mu^p(i)}|\al_j\>}g_{ij}(w/z)\right)E_-(\al_j,w)$,

$\te{(3) }[\al_i^-(z),E_+(\al_j,w)]=\frac{1}{q-q\inverse}\te{log}
    \left(q^{-\<\sum\al_{\mu^p(j)|}\al_i\>}g_{ji}(z/w)\right)E_+(\al_j,w)$,

$\te{(4) } [E_\pm(\al_i,z), E_\pm(\al_j,w)]=0$,

$\te{(5) }E_+(\al_i,z)E_-(\al_j,w)=E_-(\al_j,w)E_+(\al_i,z)\kprod
\left(
    1-\frac{\xi^kw}{z}
\right)_{q^2}^{\<\al_i|\mu^k \al_j\>}.$\\
\end{lem}
\begin{proof} The assertion (1) is clear. Note that, if $ [A,[A,B]]=0=[B,[A,B]]$, then
\begin{eqnarray}\label{eq:ecom}
Ae^B=[A,B]e^BA,\quad e^Ae^B=e^Be^Ae^{[A,B]}.
\end{eqnarray}
Then the assertions (2), (3), (4) and (5) follow from (H1), \eqref{eq:ecom}
and Lemma \ref{lem:qibinormal}.
\end{proof}

The following is easy to check.
\begin{lem}\label{lem:opecom} For every $\al,\beta\in Q$, one has
\[e_\al e_\be=\,C(\al,\be)e_\be e_\al,\
z^{\al_{(0)}} e_\be=\,z^{\<\al_{(0)}|\be\>}e_\be z^{\al_{(0)}},\
 c^{\al_{(0)}} e_\be=\,c^{\<\al_{(0)}|\be\>}e_\be c^{\al_{(0)}}.\]
\end{lem}

The commutation relations between $\al_i^{\pm}(z)$ and  $X_i^\pm(z)$ are given in the following result.

\begin{lem}\label{lem:alxcom} For any $i,j\in I$, one has

$\te{(1) } X^\pm_{\mu(i)}(z)=X^\pm_i(\xi\inverse z),$

$\te{(2) }[\al_i^{+}(z),X^{\pm}_j(w)]=\pm \frac{1}{q-q\inverse}
    \te{log}
        \left(q^{-\<\sum\mu^p\al_i|\al_j\>}
        g_{ij}(q^{\mp\half}w/z)\right)
    X^\pm_j(w)$,

$\te{(3) }[\al_i^{-}(z),X^{\pm}_j(w)]=\pm \frac{1}{q-q\inverse}
    \te{log}
        \left(q^{-\<\sum\mu^p\al_i|\al_j\>}
            g_{ji}(q^{\mp\half}z/w)\right)
        X^{\pm}_j(w)$.
\end{lem}

\begin{proof} The assertion (1) follows from Lemma  \ref{lem:alcom} (1) and the fact that
\[\hat{\mu}(e_\al)=\,e_\al \xi^{-\al_{(0)}-\<\al_{(0)}|\al\>/2},\]
which can be deduced from \eqref{eq:hatmuOnMod}. The assertions
(2) and (3) are respectively implied by Lemma \ref{lem:alcom} (2) and (3).
\end{proof}

In the following we are going to check the quantum commutate relation (Q10).
For this purpose, we need to deduce  a $q$-analogue of the identity given in \cite[Proposition 4.1]{Lep} as follows.

\begin{prop}\label{prop:xydeltafunc} Let $i,j\in I$. Then
\begin{eqnarray*}
&&\kprod\left(
        1-\frac{\xi^kw}{z}
    \right)_{q^2}^{-\<\al_i|\mu^k\al_j\>}
-
\kprod\left(
        -\frac{\xi^kw}{z}+1
    \right)_{q^2}^{-\<\al_i|\mu^k\al_j\>}\\
&=&\left\{
\begin{array}{ll}
    \ksum\delta_{i,\mu^k(j)}
    \left(
    \frac{1+q^{-d_{ii}} }{d_i(1-q^{-2d_i})}
    \delta\left(
        \frac{
            q\xi^kw
        }{
            z
        }
    \right)
+
    \frac{1+q^{d_{ii}} }{d_i(1-q^{2d_i})}
    \delta\left(
        \frac{
            q\inverse\xi^kw
        }{
            z
        }
    \right)
    \right),&\te{if }d_{ii}>0,\\
    \ksum\delta_{i,\mu^k(j)}
    \left(
    \frac{1}{d_i(1-q^{-2d_i})}
    \delta\left(
        \frac{
            q\xi^kw
        }{
            z
        }
    \right)
+
    \frac{1}{d_i(1-q^{2d_i})}
    \delta\left(
        \frac{
            q\inverse\xi^kw
        }{
            z
        }
    \right)
    \right),&\te{if }d_{ii}=0.
\end{array}
    \right.
\end{eqnarray*}
\end{prop}
\begin{proof}  For each $n\in \Z$, set $J(n)=\{k\in\Z_N\,|\,\<\al_i|\mu^k\al_j\>=n\}$.
If  $i\not\in\mathcal{O}(j)$, then $\Z_N=J(0)\cup J(-1)$ and the assertion is clear.
Suppose now that $i\in\mathcal{O}(j)$. Then $\Z_N=J(0)\cup J(-1)\cup J(2)$ with
\begin{align}\label{eq:J2}
J(2)=\{k\in\Z_N\,|\,\mu^k(j)=i\}.\end{align}
Furthermore, from Lemma \ref{lem:s-i=2}, we have that
 \begin{eqnarray}\label{eq:decom}
 \left(
        1-\frac{\xi^kw}{z}
    \right)_{q^2}^{-\<\al_i|\mu^k\al_j\>}=
    \left\{
    \begin{array}{ll}
    \displaystyle\frac{1+\frac{w^{d_{ii}}}{z^{d_{ii}}} }{
    \(1-q^{d_i}\frac{w^{d_i}}{z^{d_i}} \)\(1-q^{-d_i}\frac{w^{d_i}}{z^{d_i}} \)},
        &\te{if }d_{ii}>0,\\
    \displaystyle\frac{1}{\(1-q^{d_i}\frac{w^{d_i}}{z^{d_i}} \)
        \(1-q^{-d_i}\frac{w^{d_i}}{z^{d_i}} \)},
        &\te{if }d_{ii}=0.
    \end{array}\right.
 \end{eqnarray}


Take $k_0\in J(2)$.
 Then it follows from \eqref{eq:decom}  that
\begin{eqnarray}\label{eq:J-tttttt}
&&\left.\left(
        1-\frac{\xi^kw}{z}
    \right)_{q^2}^{-\<\al_i|\mu^k\al_j\>}(1-\frac{\xi^{k_0}q^{\pm 1}w}{z})
        \right|_{w=\xi^{-k_0}q^{\mp 1}z}\nonumber\\
   &=&
   (1+\xi^{-k_0d_{ii}}q^{\mp d_{ii}})
     (1-\xi^{-k_0d_i} q^{\mp 2 d_i})\inverse
    \lim\limits_{z\longrightarrow 1}\frac{1-z}{1-\xi^{-k_0d_i}z^{d_i}}\nonumber\\
   &=&
   (1+q^{\mp d_{ii}})
     (1-q^{\mp 2 d_i})\inverse
    \lim\limits_{z\longrightarrow 1}\frac{1-z}{1-z^{d_i}}\nonumber\\
&=&
\frac{1}{d_i}
    (1+q^{\mp d_{ii}})
    (1-q^{\mp 2d_i})\inverse,\qquad\te{if }d_{ii}>0,
\end{eqnarray}
and that
\begin{eqnarray}\label{eq:J-tttttt=0}
&&\left.\left(
        1-\frac{\xi^kw}{z}
    \right)_{q^2}^{-\<\al_i|\mu^k\al_j\>}(1-\frac{\xi^{k_0}q^{\pm 1}w}{z})
        \right|_{w=\xi^{-k_0}q^{\mp 1}z}\nonumber\\
   &=&
   (1-\xi^{-k_0d_i} q^{\mp 2 d_i})\inverse
    \lim\limits_{z\longrightarrow 1}\frac{1-z}{1-\xi^{-k_0d_i}z^{d_i}}\nonumber\\
   &=&
   (1-q^{\mp 2 d_i})\inverse
    \lim\limits_{z\longrightarrow 1}\frac{1-z}{1-z^{d_i}}\nonumber\\
&=&
\frac{1}{d_i}
    (1-q^{\mp 2d_i})\inverse,\qquad\te{if }d_{ii}=0.
\end{eqnarray}
Finally, the proposition follows from \eqref{eq:J2}, \eqref{eq:decom}, \eqref{eq:J-tttttt}, \eqref{eq:J-tttttt=0} and the next lemma.
\end{proof}

 \begin{lem}\cite[Proposition 2.8]{CGJT} Let $c_i$, $1\le i\le t$ be distinct nonzero
 complex numbers, and let $a_i\ge -1, 1\le i\le t$ be some integers.  Then one has
 \begin{align*}
 \prod^n_{i=1}\(1-\frac{c_iw}{z}\)^{a_i}-\prod^n_{i=1}\(-\frac{c_iw}{z}+1\)^{a_i}
 =\sum_{1\le i\le t;a_i=-1}\left(\prod_{1\le j\ne i\le t}(1-c_ic_j^{-1})^{a_j} \right)\delta\(\frac{c_iw}{z}\).
 \end{align*}
\end{lem}
For any $i,j\in \Z$, set
\begin{eqnarray*}
&&:X^+_i(z)X^-_j(w):
    =E_-(\al_i,q^{-\half}z)E_-(\al_j,q^{\half}w)
    E_+(\al_i,q^{\half}z)E_+(\al_j,q^{-\half}w)\nonumber\\
&&\quad\quad\times
    e_{\al_i}e_{-\al_j}
    z^{\al_{i(0)}}w^{-\al_{j(0)}}
    z^{\<\al_{i(0)}|\al_i-2\al_j\>/2}w^{\<\al_{j(0)}|\al_j\>/2}.
\end{eqnarray*}
Then one can conclude form Lemma \ref{lem:alcom} (5) and Lemma  \eqref{lem:opecom} that
\begin{align}
&\quad [X_i^+(z),X^-_j(w)]\label{eq:precomm}\\
=&:X^+_i(z)X^-_j(w):\left(\kprod\left(1-\frac{\xi^kw}{z}\right)_{q^2}^{-\<\al_i|\mu^k\al_j\>}
        -
        \kprod\left(-\frac{\xi^kw}{z}+1\right)_{q^2}^{-\<\al_i|\mu^k\al_j\>}
    \right).\nonumber
\end{align}

Recall the numbers $\epsilon_i$, $i\in I$ defined in \eqref{eplsioni}.
\begin{prop}\label{prop:qcartanxytoh}
Let $i,j\in I$. Then
\begin{eqnarray*}
[X^+_i(z),X^-_j(w)]=
   \ksum\delta_{i,\mu^k(j)}\epsilon_i^2
        \(\Phi^+_i(q^{-\half}z)
        \delta\left(
        \frac{
            q\xi^kw
        }{
            z
        }
    \right)
-
    \Phi^-_i(q^{\half}z)
    \delta\left(
        \frac{
            q\inverse\xi^kw
        }{
            z
        }\right)\).
\end{eqnarray*}
\end{prop}
\begin{proof} We only prove the assertion for the case that $d_{ii}>0$.
The case of $d_{ii}=0$ is similar and omitted.
Now one can conclude from Proposition \ref{prop:xydeltafunc} and \eqref{eq:precomm} that
\begin{align*}
&[X_i^+(z),X^-_j(w)]=\sum_{k\in \Z}\delta_{i,\mu^k(j)}:X^+_i(z)X^-_j(w):
   \\
    &\times\left(
    \frac{ 1+q^{-d_{ii}} }{d_i(1-q^{-2d_i})}
    \delta\left(
        \frac{q\xi^kw}{  z  }
    \right)
+ \frac{ 1+q^{d_{ii}} }{d_i(1-q^{2d_i})}
    \delta\left(
        \frac{
            q\inverse\xi^kw}{ z }
    \right)
    \right).
    \end{align*}
Comparing the above identity with the one given in the proposition, it suffices to check that for each $k\in J(2)$, the following holds true
\begin{align}\label{eq:commutate0}\left.
\frac{1+q^{\pm d_{ii}}}{1-q^{\pm2d_i}}:X^+_i(z)X^-_j(w):\right|_{w=\xi^{-k}q^{\pm 1}z}
=\mp\frac{q^{d_{ii}/2}+q^{-d_{ii}/2}  }{q^{d_{i}}-q^{-d_{i}}}\Phi_i^\mp(q^{\mp\half}z).
\end{align}

Indeed, it follows from Lemma \ref{lem:alxcom} (1) that
\begin{align}\label{eq:normalreplace}
&:X^+_i(z)X^-_i(\xi^{-k}q^{\pm 1}z):=:X^+_i(z)X^-_i(q^{\pm 1}z):\nonumber\\
=&
q^{\mp\al_i(0)\pm \<\al_i(0)|\al_i\>/2}
E_\mp(\al_i,q^{\mp\half}z)E_\mp(-\al_i,q^{\pm\frac{3}{2}}z).
\end{align}
Notice that
\begin{eqnarray*}
&&E_\mp(\al,q^{\mp\half}z)E_\mp(-\al,q^{\pm\frac{3}{2}}z)
=\te{exp}\left(
    \pm\sum\limits_{n>0}\frac{\al(\mp n)}{[n]_q}z^{\pm n}
    \left(
        q^{-\half n}-q^{\frac{3}{2}n}
    \right)
\right)\\&=&
    \te{exp}\left(
    \pm\sum\limits_{n>0}\frac{\al(\mp n)}{[n]_q}(q^{\pm\half}z)^{\pm n}
    \left(
        q^{- n}-q^{n}
    \right)
\right)\quad=
\te{exp}\left(
    \mp(q-q\inverse)\al^\mp(q^{\pm\half}z)
\right).
\end{eqnarray*}
Combining this with \eqref{eq:normalreplace}, one immediately gets
\begin{eqnarray*}
:X^+_i(z)X^-_i(\xi^{-k}q^{\pm 1}z):
    =q^{\pm d_{i}\mp d_{ii}/2}\Phi^\mp_i(q^{\pm\half}z),
\end{eqnarray*}
which is  equivalent to \eqref{eq:commutate0}.
\end{proof}

The verification of the relations (Q8)-(Q10) needs the
notion of normal ordered product which we now explain as follows.
For any $i_1,i_2,\cdots,i_n\in I$, set
\begin{eqnarray*}\label{eq:defnormalordeing}
&&:X^\pm_{i_1}(z_1)X^\pm_{i_2}(z_2)\cdots X^\pm_{i_n}(z_n):
    =E_-(\pm\al_{i_1},q^{\mp\half}z_1)\cdots E_-(\pm\al_{i_n},q^{\mp\half}z_n)\\
&&E_+(\pm\al_{i_1},q^{\pm\half}z_1)\cdots E_+(\pm\al_{i_n},q^{\pm\half}z_n)
    e_{\pm \al_{i_1}}\cdots  e_{\pm \al_{i_n}}
    \prod\limits_{j=1}^nz_j^{\pm\al_{i_j(0)}+\<\al_{i_j(0)}|\sum_{k=1}^n\al_{i_k}\>/2}.\nonumber
\end{eqnarray*}
The following is the operator product expansion (OPE) for multiple vertex operators. 
\begin{lem}\label{lem:serreRelation0000} For each $n\in \Z$, we have
\begin{eqnarray*}
&&X^\pm_i(z_1)\cdots X^\pm_i(z_r)X^\pm_j(w)
    X^\pm_i(z_{r+1})\cdots X^\pm_i(z_n)
=:X^\pm_i(z_1)\cdots X^\pm_i(z_n)X^\pm_j(w):\\
 && \times \prod_{1\le s<t\le n} (z_sz_t)^{-\<\al_i(0)|\al_i+\al_j\>/2}
w^{-\<\al_i(0)|\al_j\>}\kprod\Bigg(\prod_{1\le s<t\le n}
    (z_s-q^{\mp 1}\xi^kz_t)_{q^2}^{\<\al_i|\mu^k\al_i\>}\\
&&\times \prod\limits_{1\leq s\leq r}
            (z_s-q^{\mp 1}\xi^kw)_{q^2}^{\<\al_i|\mu^k\al_j\>}
        \prod\limits_{s>r}
            (-\xi^kw+q^{\mp 1}z_s)_{q^2}^{\<\al_i|\mu^k\al_j\>}
    \Bigg).
\end{eqnarray*}
\end{lem}
\begin{proof} It follows from Lemma \ref{lem:alcom} (5) and Lemma \ref{lem:opecom} that
\begin{eqnarray*}
&&\qquad X^\pm_i(z_1)\cdots X^\pm_i(z_r)X^\pm_j(w)
    X^\pm_i(z_{r+1})\cdots X^\pm_i(z_n)\\
 &&=:X^\pm_i(z_1)\cdots X^\pm_i(z_r)X^\pm_j(w)
    X^\pm_i(z_{r+1})\cdots X^\pm_i(z_n):\prod_{1\le s<t\le n} (z_sz_t)^{-\<\al_i(0)|\al_i+\al_j\>/2}
    \\
 &&\qquad\times\prod_{1\le s\le n}(z_s w)^{-\<\al_i(0)|\al_j\>/2}
\kprod\Bigg(\prod_{1\le s<t\le n}
    (z_s-q^{\mp 1}\xi^kz_t)_{q^2}^{\<\al_i|\mu^k\al_i\>}\\
&&\qquad\times\prod\limits_{s\leq r}
            (z_s-q^{\mp 1}\xi^kw)_{q^2}^{\<\al_i|\mu^k\al_j\>}
        \prod\limits_{r<t}
            (w-\xi^kq^{\mp 1}z_t)_{q^2}^{\<\al_j|\mu^k\al_i\>}
    \Bigg).
\end{eqnarray*}
Then the lemma is proved by using Lemma \ref{lem:opecom} and the fact that
\[\ C(\al_j,\al_i)\kprod(w-\xi^kq^{\mp 1}z_t)_{q^2}^{\<\al_j|\mu^k\al_i\>}=\kprod
(-\xi^kw+q^{\mp 1}z_t)_{q^2}^{\<\al_i|\mu^k\al_j\>}\]
for $i, j\in I$.
\end{proof}

Using the previous lemma,
it is easy to see that the relation (Q8) holds true. Namely, one has the following result.
\begin{lem}\label{lem:qlocal}
$F^\pm_{ij}(z,w)X^\pm_i(z)X^\pm_j(w)
=G^\pm_{ij}(z,w)X^\pm_j(w)X^\pm_i(z).$
\end{lem}

Now we turn to the proof of the quantum Serre relations (Q9)-(Q10).
For the relation (Q9), we  start with the following identity.
\begin{lem}\label{lem:ps0} The following holds true
\begin{align*}\label{eq:ps0}
&(z_{\sigma(1)}-q^{-2}z_{\sigma(2)})\big((z_{\sigma(1)}-q^{-1}w)^{-1}(z_{\sigma(2)}-q^{-1}w)^{-1}
+(q+q^{-1})(z_{\sigma(1)}-q^{-1}w)^{-1}\\
&(w-q^{-1}z_{\sigma(2)})^{-1}+(w-q^{-1}z_{\sigma(1)})^{-1}(w-q^{-1}z_{\sigma(2)})^{-1}\big)-(z_1\leftrightarrow z_2)=0\notag.
\end{align*}
\end{lem}
\begin{proof}
Set
\begin{align*}
  P(z_1,z_2,w)&=(z_1-q^{-2}z_2)\big((z_1-q^{-1}w)^{-1}(z_2-q^{-1}w)^{-1}
+(q+q^{-1})(z_1-q^{-1}w)^{-1}\\
\times(&-q^{-1}z_2+w)^{-1}+(-q^{-1}z_1+w)^{-1}(-q^{-1}z_2+w)^{-1}\big)\in\C((z_1,z_2))((w)).
\end{align*}
Note that
\begin{align*}
&(z_1-q^{-2}z_2)\big((z_1-q^{-1}w)^{-1}(z_2-q^{-1}w)^{-1}
+(q+q^{-1})(z_1-q^{-1}w)^{-1}\\
&\quad\times(w-q^{-1}z_2)^{-1}+(w-q^{-1}z_1)^{-1}(w-q^{-1}z_2)^{-1}\big)
=P(z_1,z_2,w)\\
+&
(q+q\inverse)w\inverse\delta\(\frac{q\inverse z_2}{w}\)
+\frac{qz_1-q\inverse z_2}{z_1-z_2}\(w\inverse\delta\(\frac{q\inverse z_1}{w}\)
+w\inverse\delta\(\frac{q\inverse z_2}{w}\)\),\\
&(z_2-q^{-2}z_1)\big((z_1-q^{-1}w)^{-1}(z_2-q^{-1}w)^{-1}
+(q+q^{-1})(z_2-q^{-1}w)^{-1}\\
&\quad\times(w-q^{-1}z_1)^{-1}+(w-q^{-1}z_1)^{-1}(w-q^{-1}z_2)^{-1}\big)
=P(z_2,z_1,w)\\
+&
(q+q\inverse)w\inverse\delta\(\frac{q\inverse z_1}{w}\)
+\frac{qz_2-q\inverse z_1}{z_1-z_2}\(w\inverse\delta\(\frac{q\inverse z_1}{w}\)
+w\inverse\delta\(\frac{q\inverse z_2}{w}\)\).
\end{align*}
This implies that the left hand-side of the identity given in the lemma equals to
\begin{eqnarray*}
P(z_1,z_2,w)-P(z_2,z_1,w).
\end{eqnarray*}
A direct verification shows that this formal series is zero, as desired.
\end{proof}

Using the above Lemma, we can verify the relation (Q9).
\begin{prop}\label{prop:qserre1}
For $i,j\in I$ such that $a_{ij}<0$ and $j\not\in\mathcal O(i)$, one has
\begin{eqnarray*}
\sum_{\sigma\in S_2}\Big\{&p_{ij}^\pm(z_{\sigma(1)},z_{\sigma(2)})\big(X_i^\pm(z_{\sigma(1)})X_i^\pm(z_{\sigma(2)})
    X_j^\pm(w)
   -[2]_{q^{d_{ij}}}
    X_i^\pm(z_{\sigma(1)})X_j^\pm(w)X_i^\pm(z_{\sigma(2)})\\
&    +X_j^\pm(w)X_i^\pm(z_{\sigma(1)})X_i^\pm(z_{\sigma(2)})\big)
\Big\}\ =0.
\end{eqnarray*}
\end{prop}
\begin{proof} First of all we have the following:
\begin{eqnarray*}
&&\label{eq:serre9}\kprod(z-q^{\mp 1}\xi^kw)_{q^2}^{\<\al_i|\mu^k\al_i\>}
    =\frac{
        (z^{d_i}-w^{d_i})
        (z^{d_i}-q^{\mp 2d_i}w^{d_i})
    }{
        (z^{d_{ii}}+q^{\mp d_{ii}}w^{d_{ii}})
    },\\
&&\kprod(z-q^{\mp 1}\xi^kw)_{q^2}^{\<\al_i|\mu^k\al_j\>}
    =(z^{d_{ij}}-q^{\mp d_{ij}}w^{d_{ij}})^{-1}.\notag
\end{eqnarray*}
These two identities and Lemma \ref{lem:serreRelation0000} then imply that
\begin{eqnarray*}
&&p^\pm_{ij}(z_1,z_2)
    X^\pm(\al_i,z_1)\cdots X^\pm(\al_i,z_r)X^\pm(\al_j,w)
    X^\pm(\al_i,z_{r+1})\cdots X^\pm(\al_i,z_2)\\
&&\quad=:X^\pm(\al_i,z_1)X^\pm(\al_i,z_2)X^\pm(\al_j,w):
    (z_1z_2)^{-\<\al_{i(0)}|\al_i+\al_j\>/2}
    w^{-\<\al_{i(0)}|\al_j\>}\\
&&\times
    \left(
        z_1^{d_i}-z_2^{d_i}
    \right)
    \left(
        z_1^{d_{ij}}-q^{\mp 2d_{ij}}z_2^{d_{ij}}
    \right)\prod\limits_{1\leq s\leq r}
        \left(
            z_s^{d_{ij}}-q^{\mp d_{ij}}w^{d_{ij}}
        \right)\inverse
    \prod\limits_{s>r}
        \left(
            -w^{d_{ij}}+q^{\mp d_{ij}}z_s^{d_{ij}}
        \right)\inverse.
\end{eqnarray*}
Therefore the LHS of the quantum Serre relation in this case is the product of
\[:X^\pm(\al_i,z_1)X^\pm(\al_i,z_2)X^\pm(\al_j,w):
    (z_1z_2)^{-\<\al_{i(0)}|\al_i+\al_j\>/2}
    w^{-\<\al_{i(0)}|\al_j\>}
    \left(
        z_1^{d_i}-z_2^{d_i}
    \right)\]
 and the following formal series
    \begin{align*}
&(z_1^{d_{ij}}-q^{-2d_{ij}}z_2^{d_{ij}})\big((z_1^{d_{ij}}-q^{-d_{ij}}w^{d_{ij}})^{-1}(z_2^{d_{ij}}-q^{-d_{ij}}w^{d_{ij}})^{-1}
+[2]_{q^{d_{ij}}}(z_1^{d_{ij}}-q^{-d_{ij}}w^{d_{ij}})^{-1}\\
&(w^{d_{ij}}-q^{-d_{ij}}z_2^{d_{ij}})^{-1}+(w^{d_{ij}}-q^{-d_{ij}}z_1^{d_{ij}})^{-1}(w^{d_{ij}}-q^{-d_{ij}}z_2^{d_{ij}})^{-1}\big)-(z_1\leftrightarrow z_2)£¬
\end{align*}
which is equal to zero by Lemma \ref{lem:ps0}.
\end{proof}

Finally we consider the quantum Serre relation (Q10).
\begin{prop}\label{prop:qserre2} If $i\in I$ with $d_{ii}>0$, then
\begin{align*}
\sum_{\sigma\in S_3} p_i(z_{\sigma(1)},z_{\sigma(2)},z_{\sigma(3)})\,
                X_i^\pm(z_{\sigma(1)})X_i^\pm(z_{\sigma(2)})X_i^\pm(z_{\sigma(3)})=0.
\end{align*}
\end{prop}
\begin{proof} By Lemmas \ref{lem:serreRelation0000}--\eqref{eq:serre9} it follows that
the Serre relation in this case is
the product of the common factor
\[
:X^\pm_i(z_1)X^\pm_i(z_2)X_i^\pm(z_3): (z_1z_2z_3)^{-\<\al_{i(0)}|\al_i\>}\]
and the series:
\[\sum_{\sigma\in S_3}\Bigg\{
     \left(q^{\mp \frac{3}{2}d_{ii}}z_{\sigma(1)}^{d_{ii}}
        -(q^{\frac{d_{ii}}{2}}+q^{-\frac{d_{ii}}{2}})
            z_{\sigma(2)}^{d_{ii}}
        +q^{\pm \frac{3}{2}d_{ii}}
            z_{\sigma(3)}^{d_{ii}}\right)
    \prod\limits_{a<b}
    \frac{(z_{\sigma(a)}^{d_{ii}}-z_{\sigma(b)}^{d_{ii}})(z_{\sigma(a)}^{d_{ii}}
        -q^{\mp2d_{ii}}z_{\sigma(b)}^{d_{ii}})}
    {z_{\sigma(a)}^{d_{ii}}+q^{\mp d_{ii}} z_{\sigma(b)}^{d_{ii}}} \Bigg\}, \]
 which is equal to zero by
 \cite[(4.44)]{J-inv} as desired.

\end{proof}

\bigskip
\centerline{\bf Acknowledgments}
\bigskip

This work is supported by
the NSF of China under
Grant Nos. 11471268, 11501478, 11531004 and 11701183,
the China Postdoctoral Science Foundation grant 2016M602454,
the Fundamental Research Funds for the Central University grant 20720150003
as well as the Simons Foundation grant 523868.




\end{document}